\documentclass{amsart}
\usepackage{amsmath}
\usepackage{amsthm}
\usepackage{amssymb}
\usepackage{tikz, tikz-cd}
\usepackage{ytableau}

\newcommand{\sg}{\sigma}

\def\S{\mathfrak{S}}

\def\multiset#1#2{\ensuremath{\left(\kern-.3em\left(\genfrac{}{}{0pt}{}{#1}{#2}\right)\kern-.3em\right)}}

\def\syt{\operatorname{SYT}}

\def\dinv{\operatorname{dinv}}
\def\area{\operatorname{area}}

\def\wt {\operatorname{wt}}

\def\linkpoly{L}
\def\linknorm{\widetilde{\linkpoly}}
\def\fubini{\overline{\mathcal{F}}}
\def\barstat{\operatorname{bar}}
\def\fulltwist{\operatorname{FT}}
\def\braid{\operatorname{Br}}

\newtheorem{lemma}{Lemma}[section]

\newtheorem{cor}{Corollary}[section]
\newtheorem{thm}{Theorem}[section]
\newtheorem{conj}{Conjecture}[section]

\newtheorem{defn}{Definition}[section]

\title{Link homology and the nabla operator}
\author[A.\ T.\ Wilson]{Andrew Timothy Wilson}

\begin{document}

\begin{abstract}
In recent work, Elias and Hogancamp develop a recurrence for the Poincar\'e series of the triply graded Hochschild homology of certain links, one of which is the $(n,n)$ torus link. In this case, Elias and Hogancamp give a combinatorial formula for this homology that is reminiscent of the combinatorics of the modified Macdonald polynomial eigenoperator $\nabla$. We give a combinatorial formula for the homologies of all links considered by Elias and Hogancamp. Our first formula is not easily computable, so we show how to transform it into a computable version. Finally, we conjecture a direct relationship between the $(n,n)$ torus link case of our formula and the symmetric function $\nabla p_{1^n}$.
\end{abstract}

\maketitle


\section{Introduction}

We begin by establishing some notation from knot theory, following \cite{elias-hogancamp}. The remaining sections of the paper will take a more combinatorial perspective.

The \emph{braid group on $n$ strands}, denoted $\braid_n$, can be defined by the presentation
\begin{align}
\braid_n &= \left\langle \sg_1, \sg_2, \ldots, \sg_{n-1} \, | \, \sg_i \sg_{i+1} \sg_i = \sg_{i+1} \sg_i \sg_{i+1}, \, \sg_i \sg_j = \sg_j \sg_i \right\rangle
\end{align}
for all $1 \leq i \leq n-2$ and $|i-j| \geq 2$. This group can be pictured as all ways to ``braid'' together $n$ strands, where $\sg_i$ corresponds to crossing string $i+1$ over string $i$ and the group operation is concatenation. One particularly notable braid is the \emph{full twist braid} on $n$ strands, denoted $\fulltwist_n$, which can be written
\begin{align}
\fulltwist_n &=\left( (\sg_1) (\sg_2 \sg_1) \ldots (\sg_{n-1} \sg_{n-2} \ldots \sg_1) \right)^2.
\end{align}
where multiplication is left to right. We will also need an operation $\omega$ on braids which corresponds to rotation around the horizontal axis. We define $\omega$ on $\braid_n$ by $\omega(\sg_i) = \sg_i$ and $\omega(\alpha \beta) = \omega(\beta) \omega(\alpha)$. Then $\omega$ is an anti-involution on $\braid_n$. All of our braids will have the property that the string that begins in column $i$ also ends in column $i$ for all $i$; these are sometimes called \emph{perfect braids}.

Given a braid with $n$ strands, one can form a \emph{link} (i.e.\ nonintersecting collection of knots) by identifying the top of the strand that begins in position $i$ with the bottom of the strand that ends in position $i$ for $1 \leq i \leq n$. The result is called a \emph{closed braid}. Alexander proved that every link can be represented by a closed braid (although this representation is not unique) \cite{alexander}. The closure of a perfect braid is a link that consists of $n$ separate unknots linked together.

In \cite{elias-hogancamp}, Elias and Hogancamp assign a complex $C_v$ to every binary word $v$. We describe this assignment here -- see Figure \ref{fig:cv-braid} for an example. Say $v \in \{0,1\}^n$ with $|v| = m$. We begin with two braids, the full twist braid $\fulltwist_{n-m}$ and a certain recursively defined complex $K_m$ \cite{elias-hogancamp}, which sits to the right of $\fulltwist_{n-m}$. For $i = 1$ to $n$, we feed  string $i$ into the leftmost available position in $K_m$ if $v_i = 1$; otherwise, we feed string $i$ into the leftmost available position in $\fulltwist_{n-m}$. All crossings that occur are forced to be ``positive,'' i.e.\ the right strand crosses over the left strand. This induces a braid $\beta_v \in \braid_n$ that occurs before the adjacent $\fulltwist_{n-m}$ and $K_m$. The final complex $C_v$ is obtained by performing $\omega(\beta_v)$, followed by $\beta_v$, followed by the adjacent $\fulltwist_{n-m}$ and $K_m$. We note that $C_{0^n}$ is the full twist braid $\fulltwist_n$ and that the closure of this braid is the $(n,n)$ torus link. The combinatorics of other links, in particular the $(m,n)$ torus link for $m$ and $n$ coprime, has been studied by a variety of authors in recent years \cite{gors-knots, gorsky-negut}. Haglund gives an overview of this work from a combinatorial perspective in \cite{haglund-knots}.

\begin{figure}
\label{fig:cv-braid}
\begin{center}
\includegraphics[scale=1]{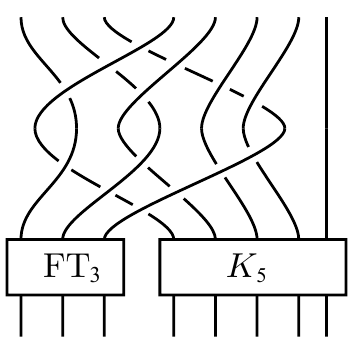}
\end{center}
\caption{We have drawn the complex $C_{10101101}$, where $\fulltwist_3$ is the full twist braid and $K_5$ is a certain complex defined recursively in \cite{elias-hogancamp}. This figure is used courtesy of \cite{elias-hogancamp}.}
\end{figure}

Elias and Hogancamp map each complex $C_v$ to a graded Soergel bimodule and then consider the \emph{Hochschild homology} of this bimodule; this is sometimes called Khovanov-Rozansky homology \cite{khovanov, khovanov-rozansky}.
This homology has three gradings: the bimodule degree (using the variable $Q$), the homological degree ($T$), and the Hochschild degree ($A$). After the grading shifts $q = Q^2$, $t = T^2 Q^{-2}$, and $a = AQ^{-2}$, Elias and Hogancamp give a recurrence for the Poincar\'e series of this triply graded homology, which they denote $f_{v}(q,a,t)$. They also give a combinatorial formula for the special case $f_{0^n}(q,a,t)$. We will give two combinatorial formulas for $f_{v}(q,a,t)$ for every $v \in \{0,1\}^n$. 

In Section \ref{sec:infinite}, we define a symmetric function $\linkpoly_v(x;q,t)$ which we call the \emph{link symmetric function}. Its definition is reminiscent of the combinatorics of the Macdonald eigenoperator $\nabla$, introduced in \cite{bght-positivity}. We prove that $f_v(q,a,t)$ is equal to a certain inner product with $\linkpoly_v(x;q,t)$. 

The main weakness of our first formula is that it is a sum over infinitely many objects, so it is not clear how to compute using this formula. We address this issue in Section \ref{sec:finite}, obtaining a finite formula for $L_v(x;q,t)$ using a collection of combinatorial objects we call \emph{barred Fubini words}.

We close by presenting some conjectures in Section \ref{sec:conjectures}. In particular, we conjecture that
\begin{align}
\linkpoly_{0^n}(x;q,t) &= (1-q)^{-n} \nabla p_{1^n} .
\end{align}
where the terminology is defined in Section \ref{sec:conjectures}.
A proof of this conjecture would provide the first combinatorial interpretation for $\nabla p_{1^n}$. There has been much recent work establishing combinatorial interpretations for $\nabla e_n$ \cite{carlsson-mellit} and $\nabla p_n$ \cite{leven-square}. We believe that $\linkpoly_v(x;q,t)$ is also related to Macdonald polynomials for general $v$, although we do not have an explicit conjecture in this direction.

\section{An infinite formula}
\label{sec:infinite}

Let $\mathbb{N} = \{0,1,2,\ldots\}$ and $\mathbb{P} = \{1,2,3,\ldots\}$. We begin by defining two statistics. 
 
\begin{defn}
Given words $\gamma \in \mathbb{N}^n$ and $\pi \in \mathbb{P}^n$, we define
\begin{align}
\area(\gamma) &= |\gamma| - \# \{1 \leq i \leq n : \gamma_i > 0 \} \\
\dinv(\gamma, \pi) &= \# \{1 \leq i < j \leq n : \gamma_i = \gamma_j, \pi_i > \pi_j \} \\
&+ \#\{1 \leq i < j \leq n : \gamma_i + 1 = \gamma_j, \pi_i < \pi_j \} \nonumber \\
x^{\pi} &= \prod_{i=1}^{n} x_{\pi_i} .
\end{align}
\end{defn}

In Figure \ref{fig:area-dinv-example}, we draw a diagram for $\gamma = 20141022$ and $\pi = 41322231$. Area counts the empty boxes in such a diagram, dinv counts certain pairs of labels, and $x^{\pi}$ records all labels that appear in the diagram. 

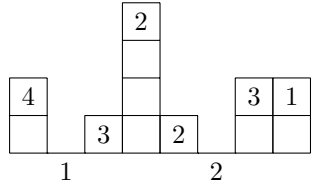
\begin{figure}
\begin{center}
\begin{tikzpicture}
\draw [step=0.5cm] (-0.001, 0) grid (0.5, 1);
\draw [step=0.5cm] (0.499, 0) grid (1, 0);
\draw [step=0.5cm] (0.999, 0) grid (1.5, 0.5);
\draw [step=0.5cm] (1.499, 0) grid (2, 2);
\draw [step=0.5cm] (1.999, 0) grid (2.5, 0.5);
\draw [step=0.5cm] (2.499, 0) grid (3, 0);
\draw [step=0.5cm] (2.999, 0) grid (3.5, 1);
\draw [step=0.5cm] (3.499, 0) grid (4, 1);

\node at (0.25, 0.75) {4};
\node at (0.75, -0.25) {1};
\node at (1.25, 0.25) {3};
\node at (1.75, 1.75) {2};
\node at (2.25, 0.25) {2};
\node at (2.75, -0.25) {2};
\node at (3.25, 0.75) {3};
\node at (3.75, 0.75) {1};

\end{tikzpicture}
\end{center}
\caption{We have depicted the example $\gamma = 20141022$ and $\pi = 41322231$ by drawing bottom-justified columns with heights $\gamma_1$, $\gamma_2$, \ldots, $\gamma_8$ and the labels $\pi_i$ are placed as high as possible in each column. In this example, we compute $\area(\gamma) = 6$, $\dinv(\gamma, \pi) = 7$, where the contributing pairs are in columns $(1, 7)$, $(1, 8)$, $(2,3)$, $(2,5)$, $(3, 5)$, $(5, 7)$, $(7, 8)$, and $x^{\pi} = x_1^2 x_2^3 x_3^2 x_4$.}
\label{fig:area-dinv-example}
\end{figure}

\begin{defn}
Given $n \in \mathbb{P}$ and $v \in \{0,1\}^n$, define 
\begin{align}
\linkpoly_{v} &= \linkpoly_{v}(x; q, t) = \sum_{\substack{\gamma \in \mathbb{N}^n, \, \pi \in \mathbb{P}^n \\ \gamma_i = 0 \Leftrightarrow v_i = 1}}  q^{\area(\gamma)} t^{\dinv(\gamma, \pi)}x^{\pi} .
\end{align}
\end{defn}

Perhaps the first thing to note about $\linkpoly_{v}$ is that it can be expressed as a sum of LLT polynomials \cite{llt}; as a result, it is symmetric in the $x_i$ variables. More precisely, each $\gamma \in \mathbb{N}^n$ can be associated with an $n$-tuple $\mathbf{\lambda}(\gamma)$ of single cell partitions in the plane, where the $i$th cell is placed on diagonal $\gamma_i$ and the order is not changed. Using the notation of \cite{hhl}, the unicellular LLT polynomial $G_{\mathbf{\lambda}(\gamma)}(x; t)$ can be used to write
\begin{align}
\linkpoly_{v} &= \sum_{\substack{\gamma \in \mathbb{N}^n \\ \gamma_i = 0 \Leftrightarrow v_i = 1}} q^{\area(\gamma)} G_{\mathbf{\lambda}(\gamma)}(x; t) .
\end{align}
Since LLT polynomials are symmetric, every $L_v$ is also symmetric.

We also remark that $\linkpoly_{1^n}$ is equal to the modified Macdonald polynomial $\widetilde{H}_{1^n}(x;q,t)$, which is also equal to the graded Frobenius series of the coinvariants of $\S_n$ with grading in $t$.

Next, we note that the Poincar\'e series $f_v(q,a,t)$ can be recovered as a certain inner product of $\linkpoly_v$. We follow the standard notation for symmetric functions and their usual inner product, as described in Chapter 7 of \cite{ec2}. Before we can prove Theorem \ref{thm:fv-lv}, we need the following lemma.

\begin{lemma}
\label{lemma:first1}
\begin{align}
\linkpoly_{0^n} &= \frac{1}{1-q} \linkpoly_{10^{n-1}} .
\end{align} 
\end{lemma}

\begin{proof}
By definition, 
\begin{align}
\linkpoly_{0^n} &= \sum_{\gamma, \pi \in \mathbb{P}^n} q^{\area(\gamma)} t^{\dinv(\gamma, \pi)} x^{\pi} .
\end{align}
Our aim is to show that
\begin{align}
\label{first1-proof}
\linkpoly_{0^n} &= q^n \linkpoly_{0^n} + \left(1 + q + \ldots + q^{n-1} \right) \linkpoly_{10^{n-1}}
\end{align}
which clearly implies the lemma. 

If $\gamma_i > 1$ for all $i$, then let $\gamma^{\prime}$ be the word obtained by decrementing each entry in $\gamma$ by 1. Set $\pi^{\prime} = \pi$. Note that the pair $(\gamma^{\prime}, \pi^{\prime})$ has
\begin{align}
\area(\gamma^{\prime}) &= \area(\gamma) - n \\
\dinv(\gamma^{\prime}, \pi^{\prime}) &= \dinv(\gamma, \pi) \\
x^{\pi^{\prime}} = x^{\pi} .
\end{align}
Furthermore, every pair of words of positive integers can be obtained as $(\gamma^{\prime}, \pi^{\prime})$ in this fashion. This case corresponds to the first term on the right-hand side of \eqref{first1-proof}.

The other case we must consider is if $\gamma_i = 1$ for some $i$. Let $k$ be the rightmost position such that $\gamma_k = 1$. Then we define
\begin{align}
\gamma^{\prime \prime} &= (\gamma_k - 1)(\gamma_{k+1} - 1) \ldots (\gamma_{n} - 1) \gamma_1 \gamma_2 \ldots \gamma_{k-1} \\
\pi^{\prime \prime} &= \pi_k \pi_{k+1} \ldots \pi_{n} \pi_1 \pi_2 \ldots \pi_{k-1}.
\end{align}
It is straightforward to check that
\begin{align}
\area(\gamma^{\prime \prime }) &= \area(\gamma) - (n-k) \\
\dinv(\gamma^{\prime \prime }, \pi^{\prime \prime }) &= \dinv(\gamma, \pi) \\
x^{\pi^{\prime \prime }} &= x^{\pi} .
\end{align}
Furthermore, by construction we have $\gamma^{\prime \prime }_1 = 0$ and the other entries of $\gamma^{\prime \prime }$ are greater than 0. Summing over all values of $k$ and pairs $(\gamma^{\prime \prime }, \pi^{\prime \prime })$ obtained in this way, we get the remaining terms in the right-hand side of \eqref{first1-proof}. 
\end{proof}

\begin{thm}
\label{thm:fv-lv}
For any $v \in \{0,1\}^n$,
\begin{align}
f_{v}(q, a, t) &= \sum_{d=0}^{n} \left\langle \linkpoly_{v}, e_{n-d} h_d \right\rangle a^d .
\end{align}
\end{thm}

\begin{proof}
Let us denote the right-hand side of the statement in the theorem by $\linkpoly_{v}(q, a, t)$. In \cite{elias-hogancamp}, the authors prove that $f_v(q,a,t)$ satisfies a certain recurrence. We will use their recurrence as our definition of $f_{v}(q,a,t)$. 

Given $v \in \{0,1\}^n$ and $w \in \{0,1\}^{n-|v|}$, we form a word $u \in \{0,1,2\}^n$ that depends on $v$ and $w$. We set $u_i = 1$ if $v_i = 1$. If $v_i = 0$, say that we are at the $j$th zero in $v$, counting from left to right. Then we set $u_i = 2w_j$. For example, if $v = 10110100$ and $w=0110$ then $u = 10112120$. We form a product
\begin{align}
\label{f-rec}
P_{v,w}(a, t)  = \prod_{i\, :\, v_i = 1} \left(t^{\# \{j < i \, : \, u_j = 1\} + \# \{j > i \, : \, u_j = 2\}} + a\right) .
\end{align}
Then the recurrence in \cite{elias-hogancamp} is 
\begin{align}
\label{eh-recurrence}
f_v(q,a,t) &= \sum_{w \in \{0,1\}^{n-|v|}} q^{n-|v|-|w|} P_{v,w}(a,t)  f_{w}(q,a,t) 
\end{align}
with base cases $f_{\emptyset}(q,a,t) = 1$ and $f_{0^n}(q,a,t) = (1-q)^{-1} f_{10^{n-1}}(q,a,t)$. We use this as the definition of $f_{v}(q,a,t)$. 

The goal of this proof is to show that $L_{v}(q,a,t)$ satisfies \eqref{f-rec}. As discussed in \cite{haglund-book}, taking the inner product with $e_{n-d} h_d$ can be thought of as replacing $\pi$ with a word containing $n-d$ $\underline{0}$'s and $d$ $1$'s. For the purposes of computing $\dinv(\gamma, \pi)$ we consider $\underline{0}$ to be less than itself, but we do not make this convention for 1. For example, if $\gamma = 1111$ and $\pi = \underline{0}1\underline{0}1$, we have $\dinv(\gamma, \pi) = 2$, where the two pairs we count are $(1, 3)$ and $(1, 2)$. With these definitions, we can write
\begin{align}
\linkpoly_v(q,a,t) &= \sum_{\substack{\gamma \in \mathbb{N}^n, \, \pi \in \{\underline{0}, 1\}^n \\ \gamma_i = 0 \Leftrightarrow v_i = 1}} q^{\area(\gamma)} t^{\dinv(\gamma, \pi)} a^{\# \text{1's in }\pi} .
\end{align}

Given such a word $\gamma$, we form a word $u$ by setting $u_i = 1$ if $\gamma_i = 0$, $u_i = 2$ if $\gamma_i = 1$, and $u_i = 0$ otherwise. From this word $u$ we construct another word $w \in \{0,1\}^{n-|v|}$ by scanning $u$ from left to right and appending a $1$ to $w$ whenever we see a $2$ in $u$ and appending a $0$ to $w$ whenever we see a $0$ in $u$. For example, if $\gamma =  013021$ we have $u = 120102$ and $w = 1001$. 

Now we can explain why $\linkpoly_v(q,a,t)$ satisfies \eqref{eh-recurrence}. First, we note that the $q^{n-|v|-|w|}$ term counts the contribution of empty boxes in row 1 to area. We also claim that $P_{v,w}(a,t)$ uniquely counts the contributions from dinv pairs $(i,j)$ with either $\gamma_i = \gamma_j = 0$ or $\gamma_i = 0$ and $\gamma_j = 1$. For each such pair, say that the pair \emph{projects onto} $j$ if $\gamma_i = \gamma_j = 0$ or $i$ if $\gamma_i = 0$ and $\gamma_j = 1$. Then every such pair projects onto a unique $i$ such that $\gamma_i = 0$, which is equivalent to $v_i = 1$. Furthermore, the number of pairs projecting onto a particular $i$ is 0 if $\pi_i = 1$ and
\begin{align}
\# \{j < i \, : \, \gamma_j = 0\} + \# \{j > i \, : \, \gamma_j = 1\} = \# \{j < i \, : \, u_j = 1\} + \# \{j > i \, : \, u_j = 2\}
\end{align}
if $\pi_i = \underline{0}$. Hence, $P_{v,w}(a,t)$ accounts for the contribution all such dinv pairs. By induction, $L_{w}(q,a,t)$ accounts for all other area and all other dinv pairs. The $v=0^n$ case follows from Lemma \ref{lemma:first1}.
\end{proof}

For the sake of comparison with \cite{elias-hogancamp}, we give a simplified formula that directly computes $f_v(q,a,t)$ from Theorem \ref{thm:fv-lv}. Given $\gamma \in \mathbb{N}^n$ and $1 \leq i \leq n$, let 
\begin{align}
\dinv_i(\gamma) &= \# \{j < i : \gamma_j = \gamma_i\} + \# \{j > i : \gamma_j  = \gamma_i + 1 \}.
\end{align}

\begin{cor}
\begin{align}
f_v(q,a,t) &= \sum_{\substack{\gamma \in \mathbb{N}^n \\ \gamma_i = 0 \Leftrightarrow v_i = 1}} q^{\area(\gamma)} \prod_{i=1}^{n} \left(a + t^{\dinv_i(\gamma)} \right) 
\end{align}
where, as before, $\area(\gamma) = |\gamma| - \# \{1 \leq i \leq n : \gamma_i > 0 \}$. 
\end{cor}

If $v=0^n$ and $a=0$, this is exactly Theorem 1.9 in \cite{elias-hogancamp}.

\section{A finite formula}
\label{sec:finite}

Although the combinatorial definition of $\linkpoly_{v}$ is straightforward, it is not computationally effective\footnote{There are also infinitely many $\pi \in \mathbb{P}^n$, but this problem can be rectified with standardization \cite{haglund-book}.} since it is a sum over infinitely many words $\gamma \in \mathbb{N}^n$. We rectify this issue in Theorem \ref{thm:link-finite} below. The idea is to compress the vectors $\gamma$ while altering the statistics so that the link polynomial $\linkpoly_v$ is not changed.

\begin{defn}
A word $\gamma \in \mathbb{N}^n$ is a \emph{Fubini word} if every integer $0 \leq k \leq \max(\gamma)$ appears in $\gamma$. \end{defn}

For example, $41255103$ is a Fubini word but $20141022$ is not a Fubini word, since it contains a 4 but not a 3. We call these Fubini words because they are counted by the Fubini numbers (\cite{oeis}, A000670), which also count ordered partitions of the set $\{1,2,\ldots,n\}$. We will actually be interested in certain decorated Fubini words.

\begin{defn}
Given $v \in \{0,1\}^n$, we say that a Fubini word $\gamma$ is \emph{associated} with $v$ if either
\begin{itemize}
\item $v = 0^n$ and the only zero in $\gamma$ occurs at $\gamma_1$, or
\item $v \neq 0^n$ and $\gamma_i = 0$ if and only if $v_i = 1$. 
\end{itemize}
\end{defn}

\begin{defn}
A \emph{barred Fubini word} associated with $v$ is a Fubini word $\gamma$ associated with $v$ where we may place bars over certain entries. Specifically, the entry $\gamma_j$ may be barred if 
\begin{enumerate}
\item $\gamma_j > 0$, 
\item $\gamma_j$ is unique in $\gamma$, and
\item for each $i<j$ we have $\gamma_i < \gamma_j$, i.e.\ $\gamma_j$ is a left-to-right maximum in $\gamma$. 
\end{enumerate}
We denote the collection of barred Fubini words associated with $v$ by $\fubini_v$. 
\end{defn}

For example, 
\begin{align}
\fubini_0 &= \{0\} \\
\fubini_{00} &= \{01, 0\overline{1}\} \\
\fubini_{000} &= \{011, 012, 0\overline{1}2, 01\overline{2}, 0\overline{1}\overline{2}, 021, 0\overline{2}1 \} .
\end{align}
The sequence $|\fubini_{0^n}|$ for $n \in \mathbb{N}$ begins $1, 1, 2, 7, 35, 226,\ldots$ and seems to appear in the OEIS as A014307 \cite{oeis}. One way to define sequence A014307 is that it has exponential generating function
\begin{align}
\sqrt{\frac{e^z}{2 - e^z}} .
\end{align}
This sequence is given several combinatorial interpretations in \cite{ren}. It would be interesting to obtain a bijection between $\fubini_{0^n}$ and one of the collections of objects in \cite{ren}. See Figure \ref{fig:fubini-words} for more examples of barred Fubini words.

\begin{figure}
\begin{tabular}{c|c}
$v$	& $\fubini_v$ \\\hline
111	& $000$ \\
011	& $100, \overline{1}00$\\
101	& $010, 0\overline{1}0$\\
110	& $001, 00\overline{1}$\\
001 & $110, 120, 1\overline{2}0, \overline{1}20, \overline{1}\overline{2}0, 210, \overline{2}10 $\\
010	& $101, 102, 10\overline{2}, \overline{1}02, \overline{1}0\overline{2}, 201, \overline{2}01$\\
100	& $011, 012, 0\overline{1}2, 01\overline{2}, 0\overline{1}\overline{2}, 021, 0\overline{2}1  $ \\
000	& $011, 012, 0\overline{1}2, 01\overline{2}, 0\overline{1}\overline{2}, 021, 0\overline{2}1  $
\end{tabular}
\caption{We have listed the barred Fubini words $\fubini_v$ for each $v \in \{0,1\}^3$.}
\label{fig:fubini-words}
\end{figure}

Given a barred Fubini word $\gamma$ and a word $\pi \in \mathbb{P}^n$, we modify the dinv statistic slightly:
\begin{align}
\dinv(\gamma, \pi) &= \# \{1 \leq i < j \leq n : \gamma_i = \gamma_j, \pi_i > \pi_j \} \\
&+ \#\{1 \leq i < j \leq n : \gamma_i + 1 = \gamma_j, \pi_i < \pi_j, \gamma_j \text{ is not barred} \} \nonumber
\end{align}
We also let $\barstat(\gamma)$ be the number of barred entries in $\gamma$. We have the following result.

\begin{thm}
\label{thm:link-finite}
For $v \in \{0,1\}^n$, 
\begin{align}
\linkpoly_{v} &= \sum_{\substack{\gamma \in \fubini_v \\ \pi \in \mathbb{P}^n}} q^{\area(\gamma) + \barstat(\gamma)} t^{\dinv(\gamma, \pi)} (1-q)^{-\barstat(\gamma)-\chi(v=0^n)} x^{\pi}
\end{align}
where $\chi$ of a statement is 1 if the statement is true and 0 if it is false.
\end{thm}

\begin{proof}
Assume, for now, that $v \neq 0^n$. 
Let $\fubini_v^{(0)}$ denote the set of all $\gamma \in \mathbb{N}^n$ such that $\gamma_i = 0$ if and only if $v_i = 1$. For each $1 \leq k \leq n$, let $\fubini_v^{(k)}$ be the set of vectors $\gamma \in \mathbb{N}^n$ such that
\begin{enumerate}
\item $\gamma_i = 0$ if and only if $v_i = 1$,
\item each number $0,1,2,\ldots,k$ appears in $\gamma$.
\end{enumerate}
We also allow certain entries to be barred. Specifically, $\gamma_j \in \fubini_v^{(k)}$ may be barred if
\begin{enumerate}
\item $0 < \gamma_j \leq k$, 
\item $\gamma_j$ is unique in $\gamma$, and
\item for each $i<j$ we have $\gamma_i < \gamma_j$, i.e.\ $\gamma_j$ is a left-to-right maximum in $\gamma$. 
\end{enumerate}
Note that $\fubini_v^{(n)} = \fubini_v$, and is therefore finite. For convenience, we set
\begin{align}
\wt_{\gamma, \pi} &= \wt_{\gamma, \pi}(x; q, t) = q^{\area(\gamma) + \barstat(\gamma)} t^{\dinv(\gamma, \pi)} (1-q)^{-\barstat(\gamma)} x^{\pi} .
\end{align}
where the dinv statistic is the one we defined for barred Fubini words. Our goal is to show that
\begin{align}
\label{preserve-weight}
\sum_{\substack{\gamma \in \fubini_v^{(k-1)} \\ \pi \in \mathbb{P}^n}} \wt_{\gamma, \pi} &= \sum_{\substack{\gamma \in \fubini_v^{(k)} \\ \pi \in \mathbb{P}^n}} \wt_{\gamma, \pi}  
\end{align}
for each $1 \leq k \leq n$. Then we can chain together these identities for $k = 1, 2, \ldots, n$ to obtain the desired result.

First, we remove the intersection $\fubini_v^{(k-1)} \cap \fubini_v^{(k)}$ from both summands in \eqref{preserve-weight} to obtain the equivalent statement
\begin{align}
\label{preserve-weight2}
\sum_{\substack{\gamma \in \fubini_v^{(k-1)} \setminus \fubini_v^{(k)} \\ \pi \in \mathbb{P}^n}} \wt_{\gamma, \pi} &= \sum_{\substack{\gamma \in \fubini_v^{(k)} \setminus  \fubini_v^{(k)}  \\ \pi \in \mathbb{P}^n}} \wt_{\gamma, \pi}  .
\end{align}
Now we wish to describe the $\gamma$ that appear in the left- and right-hand summands of \eqref{preserve-weight2}. $\gamma \in \fubini_v^{(k-1)}$ is not in $\fubini_v^{(k)}$ if and only if it does not contain a $k$; similarly, $\gamma \in \fubini_v^{(k)}$ is not in $\fubini_v^{(k-1)}$ if and only if it contains a single $k$ and that $k$ is barred. This allows us to rewrite \eqref{preserve-weight2} as 
\begin{align}
\label{preserve-weight3}
\sum_{\substack{\gamma \in \fubini_v^{(k-1)} \\ k \notin \gamma \\ \pi \in \mathbb{P}^n}} \wt_{\gamma, \pi} &= \sum_{\substack{\gamma \in \fubini_v^{(k)} \\ \overline{k} \in \gamma \\ \pi \in \mathbb{P}^n}} \wt_{\gamma, \pi}  .
\end{align}
Specifically, for each subset $S \subseteq \{1,2,\ldots,n\}$ we will show that 
\begin{align}
\label{preserve-weight4}
\sum_{\substack{\gamma \in \fubini_v^{(k-1)} \\ k \notin \gamma \\ \gamma_i < k \Leftrightarrow i \in S \\ \pi \in \mathbb{P}^n}} \wt_{\gamma, \pi} &= \sum_{\substack{\gamma \in \fubini_v^{(k)} \\ \overline{k} \in \gamma \\ \gamma_i < k \Leftrightarrow i \in S \\ \pi \in \mathbb{P}^n}} \wt_{\gamma, \pi}  .
\end{align}
Then summing over all $S$ will conclude the proof.

We consider the left-hand side of \eqref{preserve-weight4}. Note that there cannot be any dinv between entries $i$ and $j$ if $\gamma_i < k$ and $\gamma_j > k$. In this sense, the entries $i$ with $\gamma_i < k$ are independent of the columns $j$ with $\gamma_j > k$. This allows us to write the left-hand side of \eqref{preserve-weight4} as a product 
\begin{align}
q^{n-|S|} \linkpoly_{0^{n-|S|}} F_{v, S}
\end{align} 
where $F_{v, S}$ is a certain symmetric function that accounts for all contribution to the weights coming from columns $i \in S$. The factor of $q$ appears because each of the entries $j \notin S$ has an empty box in the diagram that is not counted by either of the other factors. Now we can use Lemma \ref{lemma:first1} to rewrite this product as 
\begin{align}
\label{link-prod}
\frac{q^{n-|S|}}{1-q} \linkpoly_{10^{n-|S|-1}} F_{v, S}. 
\end{align}
Let $m$ be the minimal index not in $S$. Our last goal is to show that the product in \eqref{link-prod} is equal to the right-hand side of \eqref{preserve-weight4}. 

We note that, by the definition of dinv for barred words, there are no dinv pairs $(i,j)$ with $i \in S$ and $j \notin S$, i.e.\ $\gamma_i < k$ and $\gamma_j \geq k$ for $\gamma$ that appear in the sum on the right-hand side of \eqref{preserve-weight4}. We also note that $\linkpoly_{10^{n-|S|-1}}$ accounts for the contribution from columns $j \notin S$ except that it does not account for the bar on $\gamma_{m}$. This bar contributes a factor of $q/(1-q)$. Now there are $q^{n-|S|-1}$ columns with an extra box; these are the columns $j \notin S$ and $j \neq m$. The same polynomial $F_{v,S}$ accounts for the contributions of columns $i \in S$. Multiplying these together, we obtain \eqref{link-prod}.

Finally, we must address the case  $v = 0^n$. In this case, we immediately use $\linkpoly_{0^n} = (1-q)^{-1} \linkpoly_{10^{n-1}}$ and then proceed as above. This is why Fubini words associated with $0^n$ have an ``extra'' zero at the beginning. This also slightly adjusts the weight of the summands, explaining the $\chi(v=0^n)$ in the statement of the theorem.
\end{proof}

As in Section \ref{sec:infinite}, we give a formula for computing $f_v(q,a,t)$ directly. Given a barred Fubini word $\gamma$, we define
\begin{align}
\dinv_i(\gamma) &= \# \{j < i : \gamma_j = \gamma_i\} + \# \{j > i : \gamma_j  = \gamma_i + 1, \text{$\gamma_j$ is not barred} \}.
\end{align}

\begin{cor}
\begin{align}
f_v(q,a,t) &= \sum_{\gamma \in \fubini_v} q^{\area(\gamma) + \barstat(\gamma)} (1-q)^{-\barstat(\gamma) - \chi(v=0^n)} \prod_{i=1}^{n} \left(a + t^{\dinv_i(\gamma)} \right) 
\end{align}
\end{cor}

\section{Conjectures}
\label{sec:conjectures}

So far, we have used the inner product $\langle \linkpoly_v, e_{n-d} h_d\rangle$ to compute $f_v(q,a,t)$; one might wonder if there is any value in studying the full symmetric function $L_v$. In this section, we conjecture that the link symmetric function $L_v$ is closely related to the combinatorics of Macdonald polynomials, hinting at a stronger connection between Macdonald polynomials and link homology. Following \cite{elias-hogancamp}, we must first define a ``normalized'' version of the link symmetric function $\linkpoly_{v}$.

\begin{defn}
\begin{align}
\linknorm_{v} = \linknorm_v(x;q,t) = (1-q)^{n-|v|} \linkpoly_{v}(x; q, t). 
\end{align}
\end{defn}

We could also define $\linknorm_{v}$ in terms of diagrams; each box that contains a number contributes an additional factor of $1-q$. Theorem \ref{thm:link-finite} implies that $\linknorm_{v}$ has coefficients in $\mathbb{Z}[q,t]$, whereas the coefficients of $\linkpoly_{v}$ are elements of $\mathbb{Z}[[q,t]]$. We conjecture that the normalized link symmetric function $\linknorm_{v}$ is closely connected to the Macdonald eigenoperators $\nabla$ and $\Delta$. 

The modified Macdonald polynomials $\widetilde{H}_{\mu}$ form a basis for the ring of symmetric functions with coefficients in $\mathbb{Q}(q,t)$. They can be defined via triangularity relations of combinatorially \cite{hhl, haglund-book}. Given a partition $\mu$, let $B_{\mu}$ be the alphabet of monomials $q^{i} t^j$ where $(i,j)$ ranges over the coordinates of the cells in the Ferrers diagram of $\mu$. We compute an example in Figure \ref{fig:bmu}.

\begin{figure}
\begin{displaymath}
\begin{ytableau}
t^2 \\
t	& qt	& q^2t  \\
1	&  q 	&  q^2	&  q^3
\end{ytableau}
\end{displaymath}
\caption{This is the Ferrers diagram of the partition $\mu = (4,3,1)$. In each cell we have written the monomial $q^i t^j$ that corresponds to the cell, yielding $B_{\mu} = \{1, q, q^2, q^3, t, qt, q^2 t, t^2 \}$.}
\label{fig:bmu}
\end{figure}

Given a symmetric function $F$ and a set of monomials $A = \{a_1, a_2, \ldots, a_n\}$, we let $F[A]$ be the result of setting $x_i = a_i$ for $1 \leq i \leq n$ and $x_i = 0$ for $i > n$. Then we define two operators on symmetric functions by setting, for $\mu \vdash n$,
\begin{align}
\Delta_F \widetilde{H}_{\mu} &= F \left[ B_{\mu} \right] \widetilde{H}_{\mu} \\
\nabla \widetilde{H}_{\mu} &= \Delta_{e_n} \widetilde{H}_{\mu}
\end{align}
and expanding linearly. Note that, for $\mu \vdash n$, $e_n[B_{\mu}]$ is simply the product of the $n$ monomials in $B_{\mu}$; we will sometime write $T_{\mu}$ for the product $e_n[B_{\mu}]$. 

\begin{conj}
\label{conj:p1^n}
\begin{align}
\label{p1^n1}
\nabla p_{1^n} &= \linknorm_{0^n} \\
\label{p1^n2}
\Delta_{e_{n-1}} p_{1^n} &= \sum_{\substack{ v \in \{0,1\}^n \\  |v| = 1}} \linknorm_{v}
\end{align}
In fact, both conjectures follow from the conjecture that
\begin{align}
\label{p1^n3}
\linknorm_{v0} &= \nabla p_1 \nabla^{-1} \linknorm_v.
\end{align}
\end{conj}

We should mention that Eugene Gorsky first noticed that the identity
\begin{align}
\sum_{a=0}^{d} \left\langle \nabla p_{1^n}, e_{n-d} h_d \right\rangle a^d &= (1-q)^n f_{0^n}(q,a,t)
\end{align}
seemed to hold and communicated this observation to the author via Jim Haglund. Gorsky's conjectured identity is a special case of Conjecture \ref{conj:p1^n}. It is also interesting to note that the operator in \eqref{p1^n3} appears in the setting of the Rational Shuffle Conjecture as $-\mathbf{Q}_{1,1}$ \cite{rational-shuffle}.

\begin{proof}
We prove that \eqref{p1^n3} implies \eqref{p1^n1} and \eqref{p1^n2}. The fact that \eqref{p1^n3} implies \eqref{p1^n1} is clear. For the second implication, consider $v \in \{0,1\}^n$ with $|v| =1$. Say $k$ is the unique position such that $v_k=1$. By \eqref{p1^n1}, $\linknorm_{0^{k-1}} = \nabla p_{1^{k-1}}$.  By definition, $\linknorm_{0^{k-1}1}$ considers $\gamma$ such that $\gamma_i = 0$ if and only if $i = k$. It follows that $\pi_k$ cannot be involved in any dinv pairs, and that $\gamma_k$ contributes no new area. Therefore
\begin{align}
\linknorm_{0^{k-1}1} &= p_1 \nabla p_{1^{k-1}}. 
\end{align}
Using \eqref{p1^n1} again, we get
\begin{align}
\label{nabla-p-expression}
\linknorm_{0^{k-1}10^{n-k}} &= \nabla p_{1^{n-k}} \nabla^{-1} p_1 \nabla p_{1^{k-1}}.
\end{align}

We define the \emph{Macdonald Pieri coefficients} $d_{\mu, \nu}$ by
\begin{align}
p_1 \widetilde{H}_{\nu} &= \sum_{\mu \leftarrow \nu} d_{\mu, \nu} \widetilde{H}_{\mu}.
\end{align}
where the sum is over partitions $\mu$ obtained by adding a single cell to $\nu$. Given a standard tableau $\tau$, let $\mu^{(i)}$ be the partition obtained by taking the cells containing $1,2,\ldots,i$ in $\tau$. Then each $\mu^{(i+1)}$ is obtained by adding a single cell to $\mu^{(i)}$. Let $d_{\tau}$ denote the product of the Macdonald Pieri coefficients
\begin{align}
d_{\tau} &= d_{\mu^{(1)}, \emptyset} d_{\mu^{(2)}, \mu^{(1)}} \ldots d_{\mu^{(n)}, \mu^{(n-1)}}.
\end{align}

Now we can express the right-hand side of \eqref{nabla-p-expression} as
\begin{align}
& \nabla p_{1^{n-k}} \nabla^{-1} p_1 \nabla \sum_{\nu \vdash k-1} \sum_{\tau \in \syt(\nu)} d_{\tau} \widetilde{H}_{\nu} \\
&= \nabla p_{1^{n-k}} \nabla^{-1} p_1\sum_{\nu \vdash k-1} \sum_{\tau \in \syt(\nu)} d_{\tau} T_{\nu} \widetilde{H}_{\nu} \\
&= \nabla p_{1^{n-k}} \sum_{\lambda \vdash k} \sum_{\tau \in \syt(\lambda)} d_{\tau} B_{\lambda}(\tau, n)^{-1} \widetilde{H}_{\lambda}
\end{align}
where by $B_{\lambda}(\tau, n)$ we mean the monomial $q^i t^j$ associated to the cell containing $n$ in $\tau$. Completing the computation, we get
\begin{align}
\sum_{\mu \vdash n} \widetilde{H}_{\mu} \sum_{\tau \in \syt(\mu)} d_{\tau} \prod_{i \neq k} B_{\mu}(\tau, i) .
\end{align}
Summing over all $k$, we obtain $\Delta_{e_{n-1}} p_{1^n}$.
\end{proof}

As an example of our conjecture, we can use Sage to compute
\begin{align}
\left\langle \nabla p_{1,1}, p_{1,1} \right\rangle &= 1 + q + t - qt.
\end{align}
This expression should equal $\left\langle \linknorm_{00}, p_{1,1}\right\rangle$ by Conjecture \ref{conj:p1^n}.  To compute this inner product using Theorem \ref{thm:link-finite}, we consider the barred Fubini words $01$ and $0\overline{1}$, each of which can receive labels $\pi = 12$ or $21$. The corresponding diagrams are

\vspace{10pt}

\begin{center}
\begin{tikzpicture}
\draw [step=0.5cm] (-0.001, 0) grid (0.5, 0);
\draw [step=0.5cm] (0.499, 0) grid (1, 0.5); 

\node at (0.25, -0.25) {1};
\node at (0.75, 0.25) {2};

\draw [step=0.5cm] (1.999, 0) grid (2.5, 0);
\draw [step=0.5cm] (2.499, 0) grid (3, 0.5); 

\node at (2.25, -0.25) {2};
\node at (2.75, 0.25) {1};

\draw [step=0.5cm] (3.999, 0) grid (4.5, 0);
\draw [step=0.5cm] (4.499, 0) grid (5, 0.5); 

\node at (4.25, -0.25) {1};
\node at (4.75, 0.25) {$\overline{2}$};

\draw [step=0.5cm] (5.999, 0) grid (6.5, 0);
\draw [step=0.5cm] (6.499, 0) grid (7, 0.5); 

\node at (6.25, -0.25) {2};
\node at (6.75, 0.25) {$\overline{1}$};
\end{tikzpicture}
\end{center}
where we have moved the bars from $\gamma_i$ to the corresponding $\pi_i$. The weights of these diagrams coming from Theorem \ref{thm:link-finite} are 
\begin{align}
\frac{t}{1-q} \qquad \quad \frac{1}{1-q} \qquad \quad  \frac{q}{(1-q)^2} \qquad \quad \frac{q}{(1-q)^2} 
\end{align}
respectively. After multiplying by the normalizing factor $(1-q)^2$ to go from $\linkpoly_{00}$ to $\linknorm_{00}$, we sum the resulting weights to get 
\begin{align}
(1-q)t + 1-q + q + q = 1 + q + t - qt
\end{align}
as desired.

After reading an earlier version of this paper, Fran\c{c}ois Bergeron contacted the author with the following additional conjectures.

\begin{conj}[Bergeron, 2016]
\label{conj:bergeron}
\begin{align}
\label{bergeron-append-0}
\linkpoly_{v0} &= \linkpoly_{1v} + q \linkpoly_{0v} \\
\label{bergeron2}
\linkpoly_{0^n} &= \sum_{v \in \{0,1\}^k} q^{n-|v|} \linkpoly_{v0^{n-k}} \\
t \left(\linkpoly_{u011v} - \linkpoly_{u101v} \right) &= \linkpoly_{u101v} - \linkpoly_{u110v} \\
\linknorm_{0^a 1^b 0^c} &= \nabla p_{1^c} \nabla^{-1} \widetilde{H}_{1^b} \nabla p_{1^a}  \\
\linkpoly_{1^a 0 1^b} &= \frac{t^a-1}{t^{a+b}-1} \left[ \nabla p_1 \nabla^{-1}, \widetilde{H}_{1^{a+b}} \right] + \widetilde{H}_{1^{a+b}} p_1 
\end{align}
where the  bracket represents the Lie bracket and operators are applied to 1 if nothing is explicitly specified. Bergeron also observed that $\linkpoly_v(x;q,1+t)$ is $e$-positive. (For more context on this last statement, see Section 4 of \cite{bergeron-open}.) 
\end{conj}

It is clear that \eqref{bergeron-append-0} implies \eqref{bergeron2}. We do not know of any other relations between these conjectures. We close with two more open questions.
\begin{enumerate}
\item Is there a Macdonald eigenoperator expression for $\linknorm_{v}$ for other $v$?  Perhaps we can use ideas from the Rational Shuffle Conjecture \cite{rational-shuffle}, recently proved by Mellit \cite{mellit-rational}.

\item Can we generalize our conjecture for $\nabla p_{1^n}$ to ``interpolate'' between our conjecture and the Shuffle Theorem \cite{carlsson-mellit}, or maybe the Square Paths Theorem \cite{leven-square}?
\end{enumerate}

\section{Acknowledgements}

The author would like to Ben Elias and Matt Hogancamp for their exciting paper and for use of Figure \ref{fig:cv-braid}; Lyla Fadali for reading an earlier draft; Jim Haglund for editing and feedback; Eugene Gorsky for his comments and for the idea that Elias and Hogancamp's work could be related to Macdonald polynomials; and Fran\c{c}ois Bergeron for Conjecture \ref{conj:bergeron} along with other helpful suggestions.

\bibliographystyle{alpha}
\bibliography{statistics}
\label{sec:biblio}

\end{document}